\documentclass[11pt,a4paper]{article}
\usepackage[margin=25mm]{geometry}
\usepackage{amsmath,amssymb,amsthm}
\usepackage{microtype}
\usepackage[colorlinks=true,linkcolor=blue,citecolor=blue,urlcolor=blue]{hyperref}
\hypersetup{
 pdftitle={Failure of Rockafellar's Sum Conjecture via an Explicit Non-(FPV) Operator},
 pdfsubject={Technical note inspired by Weifeng Yang; explicit formula and failure of FPV via domain closure}
}
\newcommand{\R}{\mathbb R}
\newcommand{\N}{\mathbb N}
\newcommand{\pair}[2]{\langle #1,#2\rangle}
\newcommand{\norm}[1]{\lVert #1\rVert}
\DeclareMathOperator{\dom}{dom}
\DeclareMathOperator{\gra}{gra}
\DeclareMathOperator{\sgn}{sgn}
\DeclareMathOperator{\inte}{int}
\newtheorem{theorem}{Theorem}
\newtheorem{lemma}[theorem]{Lemma}
\newtheorem{proposition}[theorem]{Proposition}
\newtheorem{corollary}[theorem]{Corollary}
\theoremstyle{remark}

\allowdisplaybreaks[1]
\title{Failure of Rockafellar's Sum Conjecture via an Explicit Non-(FPV) Operator}
\author{
Radu Ioan Bo\c{t}\thanks{Faculty of Mathematics, University of Vienna,
Oskar-Morgenstern-Platz 1, 1090 Vienna, Austria.
\texttt{email:} radu.bot@univie.ac.at}
}

\begin{document}
\maketitle
\vspace{-1.5em}
\begin{abstract}
Building on a construction of Weifeng Yang, we present a simplified
counterexample to Rockafellar's sum conjecture on $c_0(\N_0\times\N)$. The central
observation is geometric: the domain of a maximally monotone operator
can have a nonconvex norm closure.  In our example, this gives a short
proof that the operator is not of type (FPV), which in turn implies
the failure of the sum conjecture.  We define an explicit operator using geometric series and provide a detailed proof of its maximal monotonicity.  The example was developed with assistance from GPT-6 Astra. The purpose is to simplify and explain the construction,  not to claim an independent counterexample mechanism.
\end{abstract}

\section{Motivation and the FPV route}

The source of inspiration is Weifeng Yang's preprint \cite{Yang}.  Its construction on $c_0$ couples triangular blocks to a profile with a nonrealizable limiting sequence of block sums. The present note keeps this underlying obstruction but changes the scalar coupling and replaces the earlier path formulation by an explicit coordinate formula.  Our aim is expository: to simplify the realization, use the convexity of domain closures required by type (FPV), and apply the normal cone criterion to obtain the failure of the sum assertion.

For a real Banach space $X$, an operator $A:X\rightrightarrows X^*$ is monotone if
\[
 \pair{x-y}{a-c}\ge0
 \qquad \forall (x,a),(y,c)\in\gra A.
\]
A pair $(z,p)\in X\times X^*$ is monotonically related to $\gra A$ if
\[
 \pair{z-x}{p-a}\ge0\qquad \forall (x,a)\in\gra A.
\]
A monotone operator is maximally monotone precisely when every such pair already belongs to its graph.

A maximally monotone operator is of type (FPV) if every open convex set $U\subseteq X$ with $U\cap\dom A\ne\varnothing$ has the following property (see \cite[Theorem~44.1]{Simons})
\begin{equation}\label{eq:fpv}
 \left.
 \begin{aligned}
 z&\in U,\qquad p\in X^*,\\
 \pair{z-x}{p-a}&\ge0
 &&\text{for every }(x,a)\in\gra A\text{ with }x\in U
 \end{aligned}
 \right\}
 \;\Longrightarrow\; (z,p)\in\gra A.
\end{equation}

Rockafellar's sum theorem \cite{Rockafellar} proves maximality of $A+B$ on a reflexive Banach space under the qualification
\begin{equation}\label{eq:cq}
 \dom A\cap\inte(\dom B)\ne\varnothing.
\end{equation}
The unrestricted Banach-space sum conjecture asserts the same conclusion without reflexivity. The connection with (FPV) follows from the sufficient condition of Simons and Verona--Verona recorded in \cite{Simons}: for a maximally monotone $A$,
\begin{equation}\label{eq:criterion}
 \left.
 \begin{gathered}
 A+N_K\text{ is maximally monotone for every closed convex set }K\subseteq X\\
 \text{such that }\dom A\cap\inte K\ne\varnothing
 \end{gathered}
 \right\}
 \;\Longrightarrow\; A\text{ is of type (FPV)}.
\end{equation}
Here $N_K$ denotes the normal cone operator of the set $K$.  Consequently,  a maximally monotone operator failing (FPV) yields a nonmaximal normal cone sum satisfying \eqref{eq:cq}. 

\section{The explicit operator and the geometric idea}

Set
\[
 I=\N_0\times\N,\qquad X=c_0(I),\qquad X^*=\ell^1(I).
\]
The notation $c_0(I)$ means: for every $\varepsilon>0$, only finitely many coordinates satisfy $|x_{b,j}|\ge\varepsilon$. The norm is $\norm{x}_\infty=\sup_{b\geq 0,j \geq 1}|x_{b,j}|$, and
\[
 \pair{x}{p}=\sum_{b\ge0}\sum_{j\ge1}x_{b,j}p_{b,j},
 \qquad |\pair{x}{p}|\le\norm{x}_\infty\norm{p}_1.
\]
The coordinate vectors are denoted by $(e_{b,j})_{b \geq 0, j \geq 1}$. 

Define
\begin{equation}\label{eq:domain}
\begin{aligned}
D=\Biggl\{x\in X:\;&|x_{0,1}|>1,\\[-1mm]
&\sum_{b=1}^{\infty}\sum_{j=1}^{\infty}
\left|
4^b\left(x_{b,j}
+2\sum_{k=1}^{j-1}(-1)^{j-k}x_{b,k}\right)
+\frac{2(-1)^{j-1}}{|x_{0,1}|^b}
\right|<+\infty
\Biggr\}.
\end{aligned}
\end{equation}
For $x\in D$, let $Ax$ be the singleton whose element has coordinates
\begin{equation}\label{eq:operator}
(Ax)_{b,j}=
\begin{cases}
x_{0,1}-\sgn(x_{0,1}),&b=0,\ j=1,\\
0,&b=0,\ j\ge2,\\[1mm]
\displaystyle
4^b\left(x_{b,j}
+2\sum_{k=1}^{j-1}(-1)^{j-k}x_{b,k}\right)
+\frac{2(-1)^{j-1}}{|x_{0,1}|^b},
&b\ge1,\ j\ge1.
\end{cases}
\end{equation}
Set $Ax=\varnothing$ for $x\notin D$, thus $\dom A =D$.  Condition \eqref{eq:domain} ensures that $Ax$ lies  in $\ell^1(I)$ for every $x \in D$.  

\begin{theorem}\label{thm:main}
The operator defined by \eqref{eq:domain}--\eqref{eq:operator} has nonempty domain, is maximally monotone in $c_0(I)\times\ell^1(I)$, and is not of type (FPV). More precisely,
\[
 \{x_{0,1}:x\in D\}=(-\infty,-1)\cup(1,+\infty),
\]
and $\overline{\dom A}$ is not convex.
\end{theorem}

\section{Coordinate identities, matrices, and domain points}

\begin{lemma}\label{lem:coordinates}
Let $x\in D$. For every $b\ge1$ and $j\ge1$,  it holds
\begin{align}
 (Ax)_{b,1}&=4^b x_{b,1}+\frac{2}{|x_{0,1}|^{b}},\label{eq:start}\\
 (Ax)_{b,j}+(Ax)_{b,j+1}&=4^b(x_{b,j+1}-x_{b,j}),\label{eq:recurrence}\\
 x_{b,j}&=-4^{-b}\left((Ax)_{b,j}+2\sum_{k>j}(Ax)_{b,k}\right),\label{eq:tail}\\
 \sum_{j\ge1}(Ax)_{b,j}&=\frac{1}{|x_{0,1}|^{b}}.\label{eq:blocksum}
\end{align}
Consequently,
\begin{equation}\label{eq:normbound}
 \norm{Ax}_1\ge |x_{0,1}|-1+\frac1{|x_{0,1}|-1}\ge2.
\end{equation}
\end{lemma}
\begin{proof}
Let $b \geq 1$. Add the defining formulas for $(Ax)_{b,j}$ and $(Ax)_{b,j+1}$.  The coefficients of $x_{b,k}$ for $k<j$ cancel, since
\[
 (-1)^{j-k}+(-1)^{j+1-k}=0,
\]
proving \eqref{eq:recurrence}. Summing that identity from $j$ to $N > j+1$ gives
\[
 (Ax)_{b,j}+2\sum_{k=j+1}^N(Ax)_{b,k}+(Ax)_{b,N+1}
 =4^b(x_{b,N+1}-x_{b,j}).
\]
Since $Ax\in\ell^1(I)$ and $x\in c_0(I)$,  $x_{b,N+1} \to 0$ and $(Ax)_{b,N+1} \to 0$ as $N \to +\infty$. This proves \eqref{eq:tail}. The identity \eqref{eq:tail} at $j=1$ and \eqref{eq:start} give
\[
\sum_{k\ge1}(Ax)_{b,k} = \frac12 (Ax)_{b,1} - \frac{1}{2} 4^b x_{b,1} = \frac{1}{|x_{0,1}|^{b}}.
\]
Finally,
\begin{align*}
 \norm{Ax}_1
 &=|x_{0,1}-\sgn(x_{0,1})|+\sum_{b\ge1}\sum_{j\ge1}|(Ax)_{b,j}|\\
 &\ge |x_{0,1}|-1+\sum_{b\ge1}\frac{1}{|x_{0,1}|^{b}}
 =|x_{0,1}|-1+\frac1{|x_{0,1}|-1}\ge2.
\end{align*}
\end{proof}
For each block, \eqref{eq:tail} says, for every $b \geq 1$,
\begin{equation}\label{eq:matrixrelation}
 \begin{pmatrix}x_{b,1}\\x_{b,2}\\x_{b,3}\\\vdots\end{pmatrix}
 =-4^{-b}T
 \begin{pmatrix}(Ax)_{b,1}\\(Ax)_{b,2}\\(Ax)_{b,3}\\\vdots\end{pmatrix},
\end{equation}
where
\begin{equation}\label{eq:matrix}
 T=\begin{pmatrix}
 1&2&2&2&\cdots\\
 0&1&2&2&\cdots\\
 0&0&1&2&\cdots\\
 \vdots&\vdots&\vdots&\ddots&\ddots
 \end{pmatrix}.
\end{equation}
is bounded and linear from $\ell^1(\N)$ to $c_0(\N)$ , with $\|T\|=2$.   

For the forward formula,  for every $b \geq 1$,  it holds
\[
 \begin{pmatrix}(Ax)_{b,1}\\(Ax)_{b,2}\\(Ax)_{b,3}\\\vdots\end{pmatrix}
 =4^bH\begin{pmatrix}x_{b,1}\\x_{b,2}\\x_{b,3}\\\vdots\end{pmatrix}
 +2|x_{0,1}|^{-b}\begin{pmatrix}1\\-1\\1\\\vdots\end{pmatrix},
\]
where
\[
 H=\begin{pmatrix}
 1&0&0&0&\cdots\\
 -2&1&0&0&\cdots\\
 2&-2&1&0&\cdots\\
 -2&2&-2&1&\cdots\\
 \vdots&\vdots&\vdots&\vdots&\ddots
 \end{pmatrix}.
\]
is a lower triangular matrix having eacht row finite.

\begin{lemma}\label{lem:examples}
Every first-coordinate value $x_{0,1} \in \R$ with $|x_{0,1}|>1$ occurs at a point of $D$. One such point is obtained by setting
\begin{equation}\label{eq:examples}
 \begin{aligned}
 x_{0,j}&=0&& \forall j\ge2,\\
 x_{b,1}&=-\frac{1}{(4|x_{0,1}|)^{b}}&& \forall b\ge1,\\
 x_{b,j}&=0&& \forall b\ge1,\ j\ge2.
 \end{aligned}
\end{equation}
Its output has coordinates
\begin{equation}\label{eq:examplevalues}
 \begin{aligned}
 (Ax)_{0,1}&=x_{0,1}-\sgn(x_{0,1}),\\
 (Ax)_{0,j}&=0&& \forall j\ge2,\\
 (Ax)_{b,1}&=\frac{1}{|x_{0,1}|^{b}}&& \forall b\ge1,\\
 (Ax)_{b,j}&=0&& \forall b\ge1,\ j\ge2.
 \end{aligned}
\end{equation}
\end{lemma}
\begin{proof}
The array defined by \eqref{eq:examples} belongs to $c_0(I)$.  Direct substitution in \eqref{eq:operator} gives,  for every $b \geq 1$,
\[
 (Ax)_{b,1}=4^b\left(-\frac{1}{(4|x_{0,1}|)^{b}}\right)+2\frac{1}{|x_{0,1}|^{b}}
 =\frac{1}{|x_{0,1}|^{b}}.
\]
For every $j\ge2$ and every $b \geq 1$,  we have
\[
 (Ax)_{b,j}
 =-\frac{2(-1)^{j-1}}{|x_{0,1}|^{b}}+\frac{2(-1)^{j-1}}{|x_{0,1}|^{b}}=0.
\]
The output in \eqref{eq:examplevalues} is absolutely summable,  thus the point belongs to $D$.
\end{proof}

\section{Monotonicity}

In the proof of the monotonicity of $A$,  we will use that,  for every absolutely summable array $(q_{b,j})_{b,j\ge1}$,
\begin{equation}\label{eq:energy}
 \sum_{j\ge1}\left(q_{b,j}+2\sum_{k>j}q_{b,k}\right)q_{b,j}
 =\left(\sum_{j\ge1}q_{b,j}\right)^2 \qquad \forall b \geq 1.
\end{equation}
Consequently,
\begin{equation}\label{eq:weightedenergy}
 \sum_{b\ge1}4^{-b}\sum_{j\ge1}
 \left(q_{b,j}+2\sum_{k>j}q_{b,k}\right)q_{b,j}
 =\sum_{b\ge1}4^{-b}\left(\sum_{j\ge1}q_{b,j}\right)^2 \qquad \forall b \geq 1.
\end{equation}

Let $x,y\in D$. By \eqref{eq:tail},  for every $b \geq 1$ and $j \geq 1$, it holds
\[
 x_{b,j}-y_{b,j} = -4^{-b}\left((Ax)_{b,j}-(Ay)_{b,j}
 +2\sum_{k>j}\bigl((Ax)_{b,k}-(Ay)_{b,k}\bigr)\right).
\]
By using \eqref{eq:weightedenergy} and \eqref{eq:blocksum}, we get
\begin{equation}\label{eq:monopair}
\begin{aligned}
& \ \pair{x-y}{Ax-Ay} \\
= & \  (x_{0,1}-y_{0,1})^2 -(x_{0,1}-y_{0,1})\bigl(\sgn(x_{0,1})-\sgn(y_{0,1})\bigr)\\
& + \sum_{b \geq 1} -4^{-b}\sum_{j \geq 1} \left((Ax)_{b,j}-(Ay)_{b,j}
 +2\sum_{k>j}\bigl((Ax)_{b,k}-(Ay)_{b,k}\bigr)\right) \left((Ax)_{b,j}-(Ay)_{b,j}\right)\\
= & \  (x_{0,1}-y_{0,1})^2 -(x_{0,1}-y_{0,1})\bigl(\sgn(x_{0,1})-\sgn(y_{0,1})\bigr) - \sum_{b\ge1}4^{-b} \left(\sum_{j \geq 1} (Ax)_{b,j} - \sum_{j \geq 1} (Ay)_{b,j} \right)^2\\
= & \  (x_{0,1}-y_{0,1})^2 -(x_{0,1}-y_{0,1})\bigl(\sgn(x_{0,1})-\sgn(y_{0,1})\bigr)- \sum_{b\ge1}4^{-b} \left(\frac{1}{|x_{0,1}|^{b}} - \frac{1}{|y_{0,1}|^{b}} \right)^2.\\ 
\end{aligned}
\end{equation}
By the mean value theorem,  we obtain
\[
 \left | \frac{1}{|x_{0,1}|^{b}} - \frac{1}{|y_{0,1}|^{b}} \right |
 \le b\bigl||x_{0,1}|-|y_{0,1}|\bigr|
\]
and, since,  $\sum_{b\ge1}b^2 4^{-b}=\frac{20}{27}$,  it holds
\begin{equation}\label{eq:lipschitz}
 \sum_{b\ge1}4^{-b}\left(\frac{1}{|x_{0,1}|^{b}} - \frac{1}{|y_{0,1}|^{b}} \right)^2
 \le\frac{20}{27}\bigl(|x_{0,1}|-|y_{0,1}|\bigr)^2.
\end{equation}
If $x_{0,1}$ and $y_{0,1}$ have the same sign,  then $\bigl(|x_{0,1}|-|y_{0,1}|\bigr)^2 = \bigl(x_{0,1} - y_{0,1}\bigr)^2$, and
\[
 \pair{x-y}{Ax-Ay}\ge\frac7{27}(x_{0,1}-y_{0,1})^2\ge0.
\]
If their signs differ,  we may assume $x_{0,1} >1$ and $y_{0,1} < -1$. Thus
\begin{align*}
 \pair{x-y}{Ax-Ay}
& \geq  (|x_{0,1}|+|y_{0,1}|)^2  - 2 (|x_{0,1}|+|y_{0,1}|) -\frac{20}{27}(|x_{0,1}|-|y_{0,1}|)^2\\
 &= (|x_{0,1}|+|y_{0,1}|)(|x_{0,1}|+|y_{0,1}|-2) -\frac{20}{27}(|x_{0,1}|-|y_{0,1}|)^2\\
 &=\frac7{27}(|x_{0,1}|-|y_{0,1}|)^2
   +2|x_{0,1}|(|y_{0,1}|-1) +2|y_{0,1}|(|x_{0,1}|-1)>0.
\end{align*}
Thus $A$ is monotone.

\section{Maximal monotonicity}

Suppose $(z,p)\in X\times X^*$ satisfies
\begin{equation}\label{eq:candidate}
 \pair{z-x}{p-Ax}\ge0\qquad \forall x\in D.
\end{equation}
We prove $z\in D$ and $p=Az$. 

Fix $x\in D$.  For every $j\ge2$ and every $\lambda \in \R$, we have
\[
 x+\lambda e_{0,j}\in D \quad \mbox{and} \quad A(x+\lambda e_{0,j})=Ax.
\]
Equation \eqref{eq:candidate} becomes,  for every $j\ge2$ and every $\lambda \in \R$,
\[
 \pair{z-x}{p-Ax}- \pair{\lambda e_{0,j}}{p-Ax} = \pair{z-x}{p-Ax}- \lambda p_{0,j} - (Ax)_{0,j} =  \pair{z-x}{p-Ax}-\lambda p_{0,j}\ge0.
\]
From here we get that
\begin{equation}\label{eq:free}
 \boxed{p_{0,j}=0\qquad \forall j\ge2.}
\end{equation}\vspace{0.5cm}

Now fix $b\ge1$ and $j\ge1$.  For every $\lambda\in\R$, we have
\begin{equation}\label{eq:variation}
\begin{aligned}
 x+\lambda(e_{b,j}+e_{b,j+1})&\in D \quad \mbox{and} \quad A\bigl(x+\lambda(e_{b,j}+e_{b,j+1})\bigr) =Ax+\lambda4^b(e_{b,j}-e_{b,j+1}).
\end{aligned}
\end{equation}
Indeed, the output changes by $4^b\lambda$ at $(b,j)$ and by $-4^b\lambda$ at $(b,j+1)$,  and at  any later position $(b,k)$ the change is
\[
 2\cdot4^b\lambda\bigl((-1)^{k-j}+(-1)^{k-j-1}\bigr)=0.
\]
We apply \eqref{eq:candidate} to \eqref{eq:variation}. Using that $\pair{e_{b,j}+e_{b,j+1}}{4^b(e_{b,j}-e_{b,j+1})}=0$,  we obtain, for every $\lambda\in\R$,
\begin{align*}
 0\le{}&\pair{z-x}{p-Ax}\\
 &-\lambda\Bigl[p_{b,j}+p_{b,j+1}-(Ax)_{b,j}-(Ax)_{b,j+1} + 4^b(z_{b,j}-z_{b,j+1}-x_{b,j}+x_{b,j+1})\Bigr]\\
= & \ \pair{z-x}{p-Ax} -\lambda\Bigl[p_{b,j}+p_{b,j+1} + 4^b(z_{b,j}-z_{b,j+1})\Bigr],
\end{align*}
where the simplification of the expression in brackets follows using \eqref{eq:recurrence}.  This leads to
\begin{equation}\label{eq:candidaterecurrence}
 p_{b,j}+p_{b,j+1}=4^b(z_{b,j+1}-z_{b,j}) \qquad \forall b\ge1,\ j\ge1.
\end{equation}\vspace{0.5cm}

We sum the previous identity from $j$ to $N >j+1$ and let $N \to +\infty$.  Since $p\in\ell^1(I)$ and $z\in c_0(I)$,  this gives
\begin{equation}\label{eq:candidatetail}
 z_{b,j}=-4^{-b}\left(p_{b,j}+2\sum_{k>j}p_{b,k}\right) \qquad \forall b\ge1,\ j\ge1.
\end{equation}\vspace{0.5cm}

Absolute summability of $p$ implies
\begin{equation}\label{eq:summables}
 \sum_{b\ge1}\left|\sum_{j\ge1}p_{b,j}\right|
 \le\sum_{b,j\ge1}|p_{b,j}|\le\norm{p}_1<+\infty.
\end{equation}
For any test point $x\in D$,  from \eqref{eq:tail},  \eqref{eq:candidatetail},  \eqref{eq:free} and \eqref{eq:weightedenergy}, we obtain
\begin{align*}
0 \leq  \pair{z-x}{p-Ax}
 ={}&(z_{0,1}-x_{0,1})\bigl(p_{0,1}-x_{0,1}+\sgn(x_{0,1})\bigr)\\
& - \sum_{b\ge1}4^{-b} \sum_{j\ge1} \left(p_{b,j} - (Ax)_{b,j}  + 2\sum_{k>j}\left(p_{b,k} - (Ax)_{b,k}\right) \right) \left(p_{b,j} - (Ax)_{b,j}  \right) \\
={}&(z_{0,1}-x_{0,1})\bigl(p_{0,1}-x_{0,1}+\sgn(x_{0,1})\bigr) -\sum_{b\ge1}4^{-b}
 \left(\sum_{j\ge1}p_{b,j}-\sum_{j\ge1}(Ax)_{b,j}\right)^2\\
 ={}&(z_{0,1}-x_{0,1})\bigl(p_{0,1}-x_{0,1}+\sgn(x_{0,1})\bigr) -\sum_{b\ge1}4^{-b}
 \left(\sum_{j\ge1}p_{b,j}- \frac{1}{|x_{0,1}|^{b}}\right)^2.
\end{align*}
Consequently,
\begin{equation}\label{eq:scalarineq}
\begin{aligned}
 &\sum_{b\ge1}4^{-b}
 \left(\sum_{j\ge1}p_{b,j}-\frac{1}{|x_{0,1}|^{b}}\right)^2 \le(z_{0,1}-x_{0,1})
 \bigl(p_{0,1}-x_{0,1}+\sgn(x_{0,1})\bigr).
\end{aligned}
\end{equation}
By Lemma~\ref{lem:examples}, this inequality can be tested at every real first-coordinate value $x_{0,1}$ with $|x_{0,1}|>1$. Its other input coordinates no longer appear.\vspace{0.5cm}

Next we will prove that $|z_{0,1}|>1$. Suppose $|z_{0,1}|\le1$,  and fix $N\in\N$.  It holds
\begin{equation*}
\sum_{b=1}^N 4^{-b}
 \left(\sum_{j\ge1}p_{b,j}-\frac{1}{|x_{0,1}|^{b}}\right)^2 \le \sum_{b\ge1}4^{-b}
 \left(\sum_{j\ge1}p_{b,j}-\frac{1}{|x_{0,1}|^{b}}\right)^2 \le(z_{0,1}-x_{0,1})
 \bigl(p_{0,1}-x_{0,1}+\sgn(x_{0,1})\bigr).
\end{equation*}
First,  we let $x_{0,1}$ decrease to $1$.  Since $x_{0,1}-\sgn(x_{0,1})=x_{0,1}-1\to0$, we get
\begin{equation}\label{eq:plus}
 \sum_{b=1}^N4^{-b}\left(\sum_{j\ge1}p_{b,j}-1\right)^2
 \le(z_{0,1}-1)p_{0,1}.
\end{equation}
Now,  we let $x_{0,1}$ increase to $-1$ from below.  Since $x_{0,1}-\sgn(x_{0,1})=x_{0,1}+1\to0$,  we get
\begin{equation}\label{eq:minus}
 \sum_{b=1}^N4^{-b}\left(\sum_{j\ge1}p_{b,j}-1\right)^2
 \le(z_{0,1}+1)p_{0,1}.
\end{equation}

We multiply \eqref{eq:plus} by $\frac{1+z_{0,1}}{2} \geq 0$ and \eqref{eq:minus} by $\frac{1- z_{0,1}}{2} \geq 0$.  Adding yields
\begin{align*}
 \sum_{b=1}^N4^{-b}\left(\sum_{j\ge1}p_{b,j}-1\right)^2 \le\left[\frac{1+z_{0,1}}2(z_{0,1}-1)
       +\frac{1-z_{0,1}}2(z_{0,1}+1)\right]p_{0,1}=0.
\end{align*}
Hence $\sum_{j\ge1}p_{b,j}=1$ for every $1\le b\le N$. As $N$ is arbitrary, this holds for every $b\ge1$, contradicting \eqref{eq:summables}. We have proved
\begin{equation}\label{eq:outside}
\boxed{ |z_{0,1}|>1.}
\end{equation}\vspace{0.5cm}

By \eqref{eq:outside} and Lemma~\ref{lem:examples}, we can choose a test point with $x_{0,1}=z_{0,1}$. The right-hand side of \eqref{eq:scalarineq} is then zero. All summands on the left are nonnegative, so
\begin{equation}\label{eq:recoveredsums}
 \sum_{j\ge1}p_{b,j}=\frac{1}{|x_{0,1}|^{b}}  \qquad  \forall  b\ge1.
\end{equation}

We use again \eqref{eq:scalarineq}, namely,  
\[
 (z_{0,1}-x_{0,1})
 \bigl(p_{0,1}-x_{0,1}+\sgn(x_{0,1})\bigr)\ge0 \qquad \mbox{for all} \ x \ \mbox{with} \ |x_{0,1}|>1.
\]
Choose $0<\delta<|z_{0,1}|-1$.  First, we test at a point with $x_{0,1}=z_{0,1}-\delta$. Its first coordinate has the same sign as $z_{0,1}$, and the inequality becomes
\[
 p_{0,1}\ge z_{0,1}-\sgn(z_{0,1})-\delta.
\]
Next, we use a test point with $x_{0,1}=z_{0,1}+\delta$.  Again the sign is unchanged, and now
\[
 p_{0,1}\le z_{0,1}-\sgn(z_{0,1})+\delta.
\]
Letting $\delta\downarrow0$ proves
\begin{equation}\label{eq:recoveredscalar}
 \boxed{p_{0,1}=z_{0,1}-\sgn(z_{0,1}).}
\end{equation}\vspace{0.5cm}

Let $b \geq 1$.  Equation \eqref{eq:candidatetail} at $j=1$ can be written as
\[
 -4^b z_{b,1}=2\sum_{k\ge1}p_{b,k}-p_{b,1}.
\]
By using \eqref{eq:recoveredsums},  we obtain
\[
 p_{b,1}=4^b z_{b,1}+\frac{2}{|z_{0,1}|^{b}}.
\]
From \eqref{eq:candidaterecurrence}, we get the recursion
\[
 p_{b,j+1}=4^b(z_{b,j+1}-z_{b,j})-p_{b,j} \quad \forall j \geq 1.
\]
Induction yields
\begin{equation}\label{eq:recoveredformula}
 \boxed{p_{b,j}
 =4^b\left(z_{b,j}+2\sum_{k=1}^{j-1}(-1)^{j-k}z_{b,k}\right)
   +2(-1)^{j-1}|z_{0,1}|^{-b} \quad \forall j \geq 1.}
\end{equation}

Equation \eqref{eq:recoveredformula} is the original formula \eqref{eq:operator} evaluated at $z$ in every block $b\ge1$.  Equations \eqref{eq:free} and \eqref{eq:recoveredscalar} give exactly its block zero entries.  Since $p\in\ell^1(I)$ and $|z_{0,1}|>1$, the defining summability condition \eqref{eq:domain} holds at $z$. Therefore $z\in D$ and $Az=p$,  proving maximal monotonicity.

\section{Failure of (FPV) through the nonconvex domain closure}\label{sec:fpv}

In this section, we use a result of Simons \cite[Theorem~44.2]{Simons}, which states that if a maximally monotone operator is of type (FPV), then the norm closure of its domain is convex.  Thus a nonconvex norm closure of the domain rules out type (FPV).

\begin{proposition}\label{prop:nonconvexclosure}
For the operator defined by \eqref{eq:domain}--\eqref{eq:operator}, the set $\overline D=\overline{\dom A}$ is not convex. Consequently, $A$ is not of type (FPV).
\end{proposition}
\begin{proof}
Every point $x\in D$ satisfies $|x_{0,1}|>1$.  Therefore,  
\begin{equation}\label{eq:closureprojection}
 \overline D\subseteq\{x\in X:|x_{0,1}|\ge1\}.
\end{equation}
Choose the two points from Lemma~\ref{lem:examples} with first coordinates $2$ and $-2$:
\begin{equation}\label{eq:two-domain-points}
 x^+=2e_{0,1}-\sum_{b\ge1}\frac{1}{8^{b}}e_{b,1},\qquad
 x^-=-2e_{0,1}-\sum_{b\ge1}\frac{1}{8^{b}}e_{b,1}.
\end{equation}
Both belong to $D$; their outputs are
\[
 Ax^+=e_{0,1}+\sum_{b\ge1}\frac{1}{2^{b}}e_{b,1},\qquad
 Ax^-=-e_{0,1}+\sum_{b\ge1}\frac{1}{2^{b}}e_{b,1}.
\]
Their midpoint is
\begin{equation}\label{eq:missing-midpoint}
 \frac{x^++x^-}{2}=-\sum_{b\ge1}\frac{1}{2^{b}}e_{b,1} \notin \overline D.
\end{equation}
This proves that  $\overline D$ is not convex. This completes the proof of Theorem~\ref{thm:main}.
\end{proof}

\section{The sum conjecture consequence}\label{sec:sum}

Theorem~\ref{thm:main}, together with the contrapositive of \eqref{eq:criterion}, already implies the existence of a normal cone $N_C$ such that $A+N_C$ fails to be maximally monotone under \eqref{eq:cq}. We now identify such a normal cone and establish the nonmaximality of the sum directly.

\begin{corollary}\label{cor:sum}
Let
\[
 C=\{x\in X:x_{0,1}\ge-1\}.
\]
Then $A$ and $N_C$ are maximally monotone,
\[
 \dom A\cap\inte(\dom N_C)\ne\varnothing,
\]
and $A+N_C$ is not maximally monotone.  Thus, the sum theorem under Rockafellar's constraint qualification fails in nonreflexive Banach spaces.
\end{corollary}
\begin{proof}
The set $C$ is convex and closed, with $\inte C=\{x\in X:x_{0,1}>-1\}$.  The normal cone operator, defined by
\[
 N_C(x)=\{p\in X^*:\pair{y-x}{p}\le0\ \text{for all }y\in C\} \quad \mbox{for} \ x \in C,
\]
and $N_C(x)=\varnothing$ otherwise, is maximally monotone \cite[Theorem~A]{RockafellarSubdifferential}.

The point $x^+$ in \eqref{eq:two-domain-points} belongs to $D\cap\inte C = \dom A\cap\inte(\dom N_C)$,  thus the qualification holds. Since every $x\in D$ has $|x_{0,1}|>1$,
\[
 \dom(A+N_C)=D\cap C=\{x\in D:x_{0,1}>1\}.
\]
Thus,  at all $x \in \dom(A+N_C)$,  it holds $N_C(x)=\{0\}$ and $(A+N_C)x=Ax$. 

Let $x \in \dom(A+N_C)$.  By using again \eqref{eq:weightedenergy}, we have
\begin{align*}
 \pair{0-x}{-e_{0,1}-(A+N_C)x} = & \pair{-x}{-e_{0,1}-Ax} = -x_{0,1} +  \pair{x}{Ax}\\
=  & \ -x_{0,1} + x_{0,1}(x_{0,1}-1) -\sum_{b\ge1}4^{-b}\left(\sum_{j\ge1}(Ax)_{b,j}\right)^2\\
= & \ x_{0,1}^2 - \sum_{b\ge1}\frac{1}{4^{b} |x_{0,1}|^{2b}}\\
= & \ x_{0,1}^2-\frac1{4x_{0,1}^2-1}> 1- \frac{1}{3} = \frac23.
\end{align*}
Thus $(0,-e_{0,1})$ is monotonically related to graph of $A + N_C$.  It is not in that graph because $A0=\varnothing$.  This proves that $A+N_C$ is not maximally monotone.
\end{proof}

\section*{Disclosure of AI assistance}

GPT-6 Astra assisted with the mathematical development and writing
of this paper. All mathematical statements and proofs were
independently verified by the author.

\end{document}